\documentclass[12pt]{amsart}
\usepackage{latexsym,amsmath,amssymb,epsfig,amsthm}
\usepackage{graphicx}
\usepackage{tikz}
\usetikzlibrary{calc}
\usepackage[enableskew,vcentermath]{youngtab}
\usepackage{hyperref}
\hypersetup{colorlinks=false}
\input epsf
\newtheorem{theorem}{Theorem}[section]
\newtheorem{proposition}[theorem]{Proposition}
\newtheorem{corollary}[theorem]{Corollary}
\newtheorem{conjecture}[theorem]{Conjecture}
\newtheorem{definition}[theorem]{Definition}
\newtheorem{example}[theorem]{Example}
\newtheorem{lemma}[theorem]{Lemma}
\newtheorem{remark}[theorem]{Remark}

\newtheorem{question}[theorem]{Question}

\newcommand{\inte}{{\rm int}}
\newcommand{\link}{{\rm link}}

\newcommand{\sd}{{\rm sd}}

\newcommand{\aA}{{\mathcal A}}

\newcommand{\cC}{{\mathcal C}}

\newcommand{\iI}{{\mathcal I}}

\newcommand{\kk}{{\mathbf k}}

\newcommand{\RR}{{\mathbb R}}

\newcommand{\sm}{{\smallsetminus}}
\begin{document}
\title[Triangulations and theta polynomials]
{Triangulations of simplicial complexes and 
theta polynomials}

\author{Christos~A.~Athanasiadis}

\address{Department of Mathematics\\
National and Kapodistrian University of Athens\\
Panepistimioupolis\\ 15784 Athens, Greece}
\email{caath@math.uoa.gr}

\thanks{Supported by the Hellenic Foundation 
for Research and Innovation (H.F.R.I.) under the 
`2nd Call for H.F.R.I. Research Projects to 
support Faculty Members \& Researchers' 
(Project Number: HFRI-FM20-04537)}
\thanks{ \textit{Mathematics Subject 
Classifications}: Primary: 05E45; 
                  Secondary: 13C14, 55U10}
\thanks{ \textit{Key words and phrases}. 
Triangulation, homology ball, homology sphere, 
theta polynomial, local $h$-polynomial, 
unimodality, $\gamma$-positivity}

\begin{abstract}
An enumerative theory of triangulations of 
simplicial complexes has been developed by Stanley.
A key role in his theory is played by 
the local $h$-polynomial of a triangulation of
a simplex. This paper develops a parallel theory, 
in which the role of the local $h$-polynomial is 
played by a simpler invariant, namely the theta 
polynomial. This allows one to deduce unimodality 
and gamma-positivity properties of $h$-polynomials 
of triangulations of simplicial complexes from 
corresponding properties of theta polynomials,
which are studied here in some detail. To mention 
one concrete application, the $h$-polynomial of 
the antiprism triangulation of any simplicial 
homology sphere is shown to be gamma-positive, 
thus confirming Gal's conjecture in a new special 
case.
\end{abstract}

\maketitle

\section{Introduction}
\label{sec:intro}

A theory for the face enumeration of triangulations 
$\Delta'$ of a simplicial complex $\Delta$ was 
developed by Stanley~\cite[Part~I]{Sta92}, in order 
to study the behavior of the $h$-polynomials of such 
triangulations (see Section~\ref{sec:enu} for any 
undefined terminology) and their monotonicity 
properties. A major role in this theory is played by 
the concept of the local $h$-polynomial of the 
restriction of $\Delta'$ to a face $F \in \Delta$, 
which appears in Stanley's locality formula 
\cite[Theorem~3.2]{Sta92} (see also 
Theorem~\ref{thm:sta-locality}) as the `local 
contribution' to $h(\Delta',x)$ of $\Delta'$ at $F$.
This paper develops the basics of a parallel theory,
in which the local $h$-polynomial is replaced by a 
simpler enumerative invariant, named theta 
polynomial.

There are good reasons for doing so. Most importantly,
in the analogue of Stanley's locality formula (see 
Theorem~\ref{thm:theta-formula}), local $h$-polynomials 
are replaced by theta polynomials and the $h$-polynomials
of the links of the faces of $\Delta$ are replaced by
the $h$-polynomials of the barycentric subdivisions of
the links of these faces. The latter are known to have 
strong unimodality and $\gamma$-positivity properties. 
These can be transferred to the $h$-polynomial of 
$\Delta'$, provided that the theta polynomials of 
$\Delta'$ at faces of $\Delta$ have similar properties. 
This naturally leads to the study of theta polynomials. 
Theta polynomials do not always have nonnegative 
coefficients (although they do have symmetric 
coefficients). However, mild assumptions (see 
Proposition~\ref{prop:theta-positivity}) imply 
their nonnegativity and their theory works well if 
one is willing to focus on triangulations of $\Delta$ 
which can be roughly thought of as those having at 
least one interior vertex at every nonempty face of 
$\Delta$.

There have been strong indications in the literature
that theta polynomials (this is a terminology first 
adopted in this paper) may be of importance and 
independent interest. First, a locality type 
formula for local $h$-polynomials, in which theta 
polynomials of barycentric subdivisions appear, was 
discovered by Juhnke-Kubitzke, Murai and 
Sieg~\cite[Theorem~4.4]{KMS19} (see also 
Theorem~\ref{thm:KMS}). The special properties of 
these theta polynomials are crucial there to deduce 
that local $h$-polynomials of barycentric 
subdivisions of regular cell decompositions of 
the simplex are $\gamma$-positive. Second, theta 
polynomials play a crucial role in studying the 
real-rootedness of uniform triangulations of 
simplicial complexes \cite[Section~6]{Ath20+} (a 
class of triangulations which includes barycentric 
subdivisions), and antiprism 
triangulations in particular~\cite{ABK22}. Third, 
a unimodality result for theta polynomials 
follows from a recent theorem of Adiprasito 
and Yashfe~\cite[Theorem~50]{AY20} (see
Theorem~\ref{thm:AY20}).

The content and structure of this paper is as 
follows. Section~\ref{sec:enu} includes 
background material on the face enumeration of 
simplicial complexes and their triangulations 
(barycentric subdivisions, in particular). For
simplicity, we consider only geometric 
triangulations of simplicial complexes in this 
paper, although the theory can be suitably 
generalized to the quasi-geometric topological
(simplicial) subdivisions, considered 
in~\cite[Part~I]{Sta92}. Section~\ref{sec:theta}
defines the theta polynomial of any (simplicial)
homology ball $\Delta$ as 
\begin{equation} \label{eq:theta-def}
  \theta(\Delta,x) = h(\Delta,x) - 
                     h(\partial \Delta,x) 
\end{equation} 
and studies its main properties. The monotonicity 
of $\theta(\Delta,x)$ under triangulations of 
$\Delta$ is studied in Section~\ref{sec:monotone}. 
Sections~\ref{sec:theta} and~\ref{sec:monotone} 
include several applications of 
their main results (Theorems~\ref{thm:theta-formula} 
and~\ref{thm:theta-monotone-b}) to the unimodality 
and $\gamma$-positivity of $h$-polynomials and 
local $h$-polynomials of triangulations of 
Cohen--Macaulay simplicial complexes and simplices, 
respectively. Section~\ref{sec:flag} discusses the 
unimodality and $\gamma$-positivity of theta 
polynomials of homology balls and conjectures 
natural conditions under which $\gamma$-positivity 
holds (see Conjecture~\ref{conj:theta-gamma-positivity}). 
This conjecture is shown to be equivalent to the Link 
Conjecture for flag homology spheres, already proposed 
by Chudnovsky and Nevo~\cite{CN20} as a strengthening 
of Gal's conjecture~\cite{Ga05}, a fact which 
advocates for its validity. As an application of the
results of previous sections, Section~\ref{sec:app}
answers some questions about the $\gamma$-positivity
of antiprism triangulations, posed 
in~\cite[Section~8]{ABK22}. In particular, 
$h$-polynomials of antiprism triangulations of 
homology spheres are shown to be $\gamma$-positive; 
this verifies Gal's conjecture in a new special 
case.

\section{Triangulations and face enumeration}
\label{sec:enu}

This section includes definitions and background on
simplicial complexes, triangulations and their 
enumerative invariants which are of primary interest 
in this paper. We denote by $|V|$ the cardinality 
and by $2^V$ the power set of a finite set $V$.

\subsection{Polynomials}
For polynomials $p(x), q(x) \in \RR[x]$ we write $p(x) 
\le q(x)$ if $q(x) - p(x)$ has nonnegative coefficients. 
A polynomial $p(x) = a_0 + a_1 x + \cdots + a_n x^n \in 
\RR[x]$ is called
\begin{itemize}
\item[$\bullet$] 
  \emph{symmetric}, with center of symmetry $n/2$, if 
	$a_i = a_{n-i}$ for all $0 \le i \le n$,
\item[$\bullet$] 
  \emph{unimodal}, with a peak at position $k$, if $a_0 
	\le a_1 \le \cdots \le a_k \ge a_{k+1} \ge \cdots \ge 
	a_n$,
\item[$\bullet$] 
  \emph{$\gamma$-positive}, with center of symmetry $n/2$,
	if $p(x) = \sum_{i=0}^{\lfloor n/2 \rfloor} \gamma_i 
	x^i (1+x)^{n-2i}$ for some nonnegative real numbers 
	$\gamma_0, \gamma_1,\dots,\gamma_{\lfloor n/2 \rfloor}$,
\item[$\bullet$] 
  \emph{real-rooted}, if every root of $p(x)$ is real, or 
	$p(x) \equiv 0$.
\end{itemize}
Every $\gamma$-positive polynomial is symmetric and 
unimodal and every real-rooted and symmetric polynomial
with nonnegative coefficients is $\gamma$-positive; see
\cite{Ath18, Bra15, Sta89} for more information on the 
connections among these concepts.  

Every polynomial $p(x) \in \RR[x]$ of degree at most 
$n$ can be written uniquely in the form $p(x) = a(x) 
+ xb(x)$ for some polynomials $a(x), b(x) \in \RR[x]$ 
of degrees at most $n$ and $n-1$, which are symmetric
with centers of symmetry $n/2$ and $(n-1)/2$, 
respectively. This expression is called the 
\emph{symmetric decomposition} of $p(x)$ with respect 
to $n$; see \cite{AT21, BS20} and references therein. 
We say that this decomposition is \emph{nonnegative, 
unimodal, $\gamma$-positive} or \emph{real-rooted} if 
both $a(x)$ and $b(x)$ have the corresponding property.
We note that every polynomial which has a nonnegative
and unimodal symmetric decomposition with respect 
to $n$ is unimodal, with a peak at $\lfloor (n+1)/2
\rfloor$.

\subsection{Simplicial complexes}
\label{subsec:complexes}
All simplicial complexes we consider will be abstract 
and finite. Thus, given a finite set $\Omega$, a 
\emph{simplicial complex} on the ground set $\Omega$ 
is a collection $\Delta$ of subsets of $\Omega$ such 
that $F \subseteq G \in \Delta \Rightarrow F \in 
\Delta$. The elements of $\Delta$ are called 
\emph{faces}. The
dimension of a face $F$ is defined as one less than 
the cardinality of $F$. The dimension of $\Delta$ is 
the maximum dimension of a face and is denoted by 
$\dim(\Delta)$. Faces of $\Delta$ of dimension zero 
or one are called \emph{vertices} or \emph{edges}, 
respectively, and those which are maximal with 
respect to inclusion are called \emph{facets}. The
\emph{link} of the face $F \in \Delta$ is the 
subcomplex of $\Delta$ defined as $\link_\Delta(F) 
= \{ G \sm F: G \in \Delta, \, F \subseteq G\}$; in 
particular, $\link_\Delta(\varnothing) = \Delta$. 
	
All topological properties of $\Delta$ we mention 
will refer to those of the geometric realization 
of $\Delta$ \cite[Section~9]{Bj95}, uniquely 
defined up to homeomorphism. All homological 
properties of $\Delta$ will be considered with 
respect to a fixed field $\kk$. Thus, $\Delta$
is said to be \emph{Cohen--Macaulay} (over $\kk$) 
if
\[ \widetilde{H}_i \, (\link_\Delta (F), \kk) 
   = 0 \]
for every $F \in \Delta$ and every $i < \dim 
\link_\Delta(F)$, where $\widetilde{H}_*(\Gamma, 
\kk)$ denotes reduced simplicial homology of 
$\Gamma$ (with coefficients in $\kk$). Moreover,
$\Delta$ is called a \emph{homology sphere} 
(over $\kk$) if 
\[ \widetilde{H}_i \, (\link_\Delta (F), \kk) 
    = \begin{cases} \kk, & 
		\text{if $i = \dim \link_\Delta (F)$} \\
    0, & \text{otherwise} 
		\end{cases} \]
for every $F \in \Delta$ and every $i$. An  
$(n-1)$-dimensional simplicial complex $\Delta$ 
is called a \emph{homology ball} (over $\kk$) if 
there exists a subcomplex $\partial \Delta$ of 
$\Delta$, called the \emph{boundary} of $\Delta$, 
with the following properties:
\begin{itemize}
\item[$\bullet$] 
$\partial \Delta$ is an $(n-2)$-dimensional 
homology sphere (over $\kk$),

\item[$\bullet$] 
for every $F \in \Delta$ and every $i$,
\[ \widetilde{H}_i \, (\link_\Delta(F), \kk) = 
   \begin{cases} \kk, & 
	 \text{if $F \notin \partial \Delta$ and $i = \dim
         \link_\Delta (F)$} \\
   0, & \text{otherwise}. \end{cases} \]
\end{itemize}
The \emph{interior} of $\Delta$ is then defined as 
$\inte(\Delta) = \Delta \sm \partial \Delta$. 

We assume familiarity with basic properties of 
Cohen--Macaulay simplicial complexes 
\cite[Section~11]{Bj95} \cite[Chapter~II]{StaCCA} 
and with those of homology 
balls and spheres as explained, for instance, 
in~\cite[Section~2B]{Ath12}. We note, in 
particular, that the complex $\link_\Delta(F)$ is 
Cohen--Macaulay for every Cohen--Macaulay simplicial 
complex $\Delta$ and every $F \in \Delta$. Moreover,
if $\Delta$ is a homology ball or a homology sphere, 
then $\link_\Delta(F)$ is a homology sphere for every 
$F \in \inte(\Delta)$ and for every $F \in \Delta$, 
respectively, and if $\Delta$ is a homology
ball, then $\link_\Delta(F)$ is a homology ball with 
boundary $\partial \, \link_\Delta(F) = 
\link_{\partial \Delta}(F)$ for every $F \in \partial 
\Delta$. Every cone over a homology sphere $\Delta$
(meaning, the simplicial complex $\Delta \cup \{F 
\cup \{v\}: F \in \Delta\}$, where $v$ is a new 
vertex, not in $\Delta$) is a homology ball with 
boundary $\Delta$.
	
Following~\cite{AT21}, we use the term 
\emph{Cohen--Macaulay*} instead of 
\emph{uniformly Cohen--Macaulay} for the class of
simplicial complexes introduced and studied by 
Matsuoka and Murai~\cite{MM16}. Thus, a 
Cohen--Macaulay simplicial complex $\Delta$ (over 
$\kk$) is called \emph{Cohen--Macaulay*} (over $\kk$) 
if the simplicial complex obtained from $\Delta$ by 
removing any of its facets is Cohen--Macaulay 
(over $\kk$) of the same dimension as $\Delta$. This 
class of simplicial complexes
includes all doubly Cohen--Macaulay simplicial 
complexes~\cite[p.~94]{StaCCA} (and, in particular,
all homology spheres). 

A convenient way to record the face numbers of a 
simplicial complex $\Delta$ is the $h$-polynomial, 
defined by the formula
\begin{equation}
\label{eq:hdef}
h(\Delta, x) = \sum_{i=0}^n f_{i-1} (\Delta) \, 
x^i (1-x)^{n-i}, 
\end{equation}
where $f_i(\Delta)$ is the number of $i$-dimensional 
faces of $\Delta$ and $n-1$ is its dimension. The 
sequence $h(\Delta) = (h_0(\Delta), 
h_1(\Delta),\dots,h_n(\Delta))$ is the 
\emph{$h$-vector} of $\Delta$. The polynomial 
$h(\Delta,x)$ has nonnegative coefficients for every 
Cohen--Macaulay complex $\Delta$ (in particular, 
for homology balls and spheres). Moreover, it 
is symmetric, with center of symmetry $n/2$, if 
$\Delta$ is a homology sphere and has the property
that $x^n h(\Delta,1/x) = h(\inte(\Delta),x)$ if 
$\Delta$ is a homology ball, where $h(\inte(\Delta),x)$ 
is defined by the right-hand side of 
Equation~(\ref{eq:hdef}) if $f_{i-1}(\Delta)$ is 
replaced by the number of $(i-1)$-dimensional 
faces of $\inte(\Delta)$. In particular, for every 
homology ball $\Delta$ of dimension $n-1$ we have 
$h_n(\Delta) = 0$ and $h_{n-1}(\Delta)$ is equal to
the number of interior vertices of $\Delta$. We 
refer to~\cite{StaCCA} for the significance and 
for more information on $h$-vectors of 
Cohen--Macaulay simplicial complexes. 

\subsection{Triangulations}
By the term \emph{triangulation} of a simplicial 
complex $\Delta$ we 
will always mean a geometric triangulation. Thus, 
a simplicial complex $\Delta'$ is a triangulation 
of $\Delta$ if there exists a geometric realization 
$K'$ of $\Delta'$ which geometrically subdivides 
a geometric realization $K$ of $\Delta$. The 
\emph{carrier} of a face $F' \in \Delta'$ is the 
smallest face $F \in \Delta$ for which the face 
of $K$ corresponding to $F$ contains the face of 
$K'$ corresponding to $F'$. The restriction of 
$\Delta'$ to a face $F \in \Delta$ is a 
triangulation of the simplex $2^F$ denoted by 
$\Delta'_F$.

The \emph{local $h$-polynomial} of a triangulation 
$\Gamma$ of a simplex $2^V$ was defined by 
Stanley~\cite[Definition~2.1]{Sta92} by the formula 
\begin{equation} \label{eq:def-local-h}
  \ell_V (\Gamma, x) = \sum_{F \subseteq V} 
  \, (-1)^{|V \sm F|} \, h (\Gamma_F, x).
\end{equation}
Stanley~\cite{Sta92} showed 
that $\ell_V (\Gamma, x)$ has nonnegative
coefficients and that it is symmetric, with center
of symmetry $|V|/2$. He exploited these properties 
and the formula of the following theorem in order
to prove nontrivial results about the face 
enumeration of triangulations of a simplicial 
complex $\Delta$. We recall that $\Delta$ is 
called \emph{pure} if all its facets have the 
same dimension.
\begin{theorem} \label{thm:sta-locality} 
{\rm (\cite[Theorem~3.2]{Sta92})}
For every pure simplicial complex $\Delta$ and every 
triangulation $\Delta'$ of $\Delta$,
\begin{equation} \label{eq:sta-locality}
h (\Delta', x) = \sum_{F \in \Delta} 
\ell_F (\Delta'_F, x) \, h (\link_\Delta (F), x).
\end{equation}
\end{theorem}

Stellar subdivisions provide a simple way to 
triangulate a simplicial complex $\Delta$. Given 
a face $F \in \Delta$, the \emph{stellar 
subdivision} of $\Delta$ on $F$ (with new vertex 
$v$) is the simplicial complex obtained from 
$\Delta$ by removing all its faces containing 
$F$ and adding all sets of the form $\{v\} \cup
E \cup E'$, where $E \subsetneq F$ and $E' \in 
\link_\Delta(F)$. 

\subsection{Barycentric subdivisions}
\label{subsec:bary}
The \emph{barycentric subdivision} of a simplicial 
complex $\Delta$ is defined as the simplicial 
complex $\sd(\Delta)$ on the vertex set $\Delta 
\sm \{\varnothing\}$ whose faces are the chains 
$F_0 \subsetneq F_1 \subsetneq \cdots \subsetneq 
F_k$ of nonempty faces of $\Delta$. This simplicial 
complex is naturally a triangulation of $\Delta$; 
the carrier of the chain $F_0 \subsetneq F_1 
\subsetneq \cdots \subsetneq F_k$ is its top 
element $F_k \in \Delta$. As a result, 
$\partial (\sd(\Delta)) = \sd(\partial \Delta)$
for every homology ball $\Delta$. 

The face enumeration of $\sd(\Delta)$ was studied 
by Brenti and Welker~\cite{BW08}. We summarize 
some of their results in the form of the following
proposition. 
\begin{proposition} [\cite{BW08}] 
\label{prop:BW08}
For integers $0 \le k \le n$ there exist 
polynomials $p_{n,k}(x)$ of degree at most $n$ 
with nonnegative coefficients, such that 
\begin{equation} \label{eq:BW08a}
h (\sd(\Delta), x) = \sum_{k=0}^n h_k
(\Delta) p_{n,k}(x) 
\end{equation}
for every $(n-1)$-dimensional simplicial 
complex $\Delta$. The $p_{n,k}(x)$ satisfy the 
recurrence 
\begin{equation} \label{eq:BW08b}
p_{n,k}(x) = x \sum_{i=0}^{k-1} p_{n-1,i}(x) + 
\sum_{i=k}^n p_{n-1,i}(x) 
\end{equation}
for every $n \ge 1$ and all $0 \le k \le n$, with 
the initial condition $p_{0,0}(x) = 1$, and have 
the property that
\begin{equation} \label{eq:BW08c}
p_{n,n-k}(x) = x^n p_{n,k}(1/x)
\end{equation}
for $0 \le k \le n$.
\end{proposition}

\smallskip
For example,
\[	p_{4,k} (x) = \begin{cases}
    1 + 11x + 11x^2 + x^3, & \text{if $k=0$} \\
    8x + 14x^2 + 2x^3, & \text{if $k=1$} \\
    4x + 16x^2 + 4x^3, & \text{if $k=2$} \\
    2x + 14x^2 + 8x^3, & \text{if $k=3$} \\
		x + 11x^2 + 11x^3 + x^4, & \text{if $k=4$}. 
\end{cases} \]

\smallskip
Brenti and Welker~\cite[Theorem~2]{BW08} showed 
that $h(\sd(\Delta),x)$ is real-rooted, hence 
unimodal, for every Cohen--Macaulay simplicial
complex $\Delta$ (more generally, for every
Boolean cell complex $\Delta$ with nonnegative 
$h$-vector). The location of the peak of 
$h(\sd(\Delta),x)$ was studied by 
Kubitzke--Nevo~\cite{KN09} and 
Murai--Yanagawa~\cite{MY14} using methods of 
commutative algebra, and by Murai~\cite{Mu10} 
and Beck--Jochemko--McCullough 
\cite[Section~3]{BJM19}, using combinatorial 
methods. The following statement is a more 
detailed version of \cite[Corollary~4.7]{KN09} 
(and applies to Boolean cell complexes with 
nonnegative $h$-vector as well), which we will 
use in Section~\ref{sec:theta} to obtain a more 
general result.
\begin{proposition} \label{prop:h-sd}
Let $\Delta$ be any $(n-1)$-dimensional 
Cohen--Macaulay simplicial complex. Then, 
$h(\sd(\Delta), x)$ can be written as a sum of 
three polynomials with nonnegative, symmetric and 
unimodal coefficients and centers of symmetry 
$(n-1)/2$, $n/2$ and $(n+1)/2$, respectively.

In particular, $h(\sd(\Delta),x)$ is unimodal, 
with a peak at position $n/2$, if $n$ is even, 
and at $(n-1)/2$ or $(n+1)/2$, if $n$ is odd. 
\end{proposition}

This statement can be derived from 
\cite[Theorem~1.1]{MY14}. We offer here a 
self-contained, elementary proof which is based 
on the following lemma (the unimodality part of 
which has also appeared as 
\cite[Lemma~3.5]{BJM19}). We recall that a 
real-rooted polynomial $p(x)$, with roots 
$\alpha_1 \ge \alpha_2 \ge \cdots$, 
\emph{interlaces} a real-rooted polynomial $q(x)$, 
with roots $\beta_1 \ge \beta_2 \ge \cdots$, if
$\cdots \le \alpha_2 \le \beta_2 \le \alpha_1 \le
\beta_1$.
\begin{lemma} \label{lem:pnk}
The polynomial $p_{n,k}(x)$ has a nonnegative,
real-rooted (in particular, unimodal and 
$\gamma$-positive) symmetric decomposition with
respect to $n$, for all $n/2 \le k \le n$.
\end{lemma}

\begin{proof} 
By \cite[Theorem~2.6]{BS20}, it suffices to show
that $p_{n,k}(x)$ has a nonnegative symmetric 
decomposition with respect to $n$ and that it is
real-rooted and interlaced by $x^n p_{n,k}(1/x)$,
for $n/2 \le k \le n$. The first claim is a special 
case of \cite[Lemma~4.2]{AT21}, from which explicit 
formulas for the symmetric parts of $p_{n,k}(x)$ in 
terms of the polynomials $p_{n-1,i}(x)$ can also be 
deduced. Given that $x^n p_{n,k}(1/x) = p_{n,n-k}
(x)$ (see Proposition~\ref{prop:BW08}), the second 
claim follows from the fact that 
$(p_{n,0}(x), p_{n,1}(x),\dots,p_{n,n}(x))$ is an 
iterlacing sequence of real-rooted polynomials; see, 
for instance, \cite[Example~7.8.8]{Bra15}.
\end{proof}

\begin{proof}[Proof of Proposition~\ref{prop:h-sd}]
By Lemma~\ref{lem:pnk} and the symmetry property 
(\ref{eq:BW08c}) we know that $p_{n,k}(x)$ can 
be written as a sum of two polynomials with 
nonnegative, symmetric and unimodal coefficients
and centers of symmetry $n/2$ and $(n+1)/2$ 
(respectively, $n/2$ and $(n-1)/2$) for all 
$n/2 \le k \le n$ (respectively, $0 \le k \le 
n/2$). Given that, the proof follows from 
Equation~(\ref{eq:BW08a}) and the fact that 
$h_k(\Delta) \ge 0$ for every $k$.
\end{proof}

The local $h$-polynomial $\ell_V(\sd(2^V),x)$ 
of the barycentric subdivision of the 
$(n-1)$-dimensional simplex $2^V$ is called the 
\emph{$n$th derangement polynomial} and is denoted 
by $d_n(x)$. A simple combinatorial interpretation 
in terms of permutation enumeration was given in
\cite[Proposition~2.4]{Sta92}; see also 
\cite[Section~3.3.1]{Ath18}. The polynomial 
$d_n(x)$ is $\gamma$-positive, hence unimodal, 
for every $n$; see \cite[Section~2.1.4]{Ath18} 
and references therein.

\section{Basic properties of theta polynomials}
\label{sec:theta}

This section discusses basic properties of theta 
polynomials and some of their immediate 
consequences for $h$-polynomials and local 
$h$-polynomials.
\begin{definition} \label{def:theta}
The theta polynomial is defined as $\theta
(\Delta, x) = h(\Delta, x) - h(\partial \Delta, 
x)$ for any homology ball $\Delta$. 
\end{definition}

The formula for $\theta(\Delta,x)$ in the 
following proposition is equivalent to the one 
for $h(\partial \Delta,x)$, given in 
\cite[Lemma~2.3]{Sta93}; see also 
Section~\ref{sec:flag}. The last statement follows
directly from the definitions of $h(\Delta, x)$ 
and $h(\partial \Delta, x)$.
\begin{proposition} [\cite{Sta93}] 
\label{prop:theta-symmetry}
For every $n \ge 1$ and every 
$(n-1)$-dimensional homology ball $\Delta$, 
\begin{equation} \label{eq:theta-Sta93}
\theta(\Delta,x) = \sum_{i=1}^{n-1} 
\left( h_{n-1}(\Delta) + \cdots + h_{n-i}(\Delta) 
- h_0(\Delta) - \cdots - h_{i-1}(\Delta) \right) 
x^i . 		
\end{equation}
In particular, $\theta(\Delta,x)$ is symmetric 
with center of symmetry $n/2$, i.e., $x^n \theta
(\Delta,1/x) = \theta(\Delta,x)$. Moreover, it has 
zero constant term and the coefficient of $x$ is 
equal to one less than the number of interior 
vertices of $\Delta$.
\end{proposition}

\begin{example} \label{ex:theta} \rm
The following claims, with the exception of the 
$\dim(\Delta) = 2$ case of part (b), are easy 
consequences of Definition~\ref{def:theta}.

(a) For an $(n-1)$-dimensional simplex $\Delta = 
2^V$ we have 
\[ \theta(\Delta,x) = \begin{cases}
    1, & \text{if $n=0$} \\
    0, & \text{if $n=1$} \\
    - x - x^2 - \cdots - x^{n-1}, & 
		\text{if $n \ge 2$}, \end{cases}\]
since $h(\partial \Delta,x) = 1 + x + \cdots + 
x^{n-1}$ for $n \ge 1$. Note that $\Delta = 
\{ \varnothing \}$, for $n=0$. 

(b) We have
\[ \theta(\Delta,x) = \begin{cases}
    (r-1)x, & \text{if $\dim(\Delta) = 1$}, \\
    (r-1)(x+x^2), & \text{if $\dim(\Delta) = 2$}, 
		\end{cases} \]
where $r$ is the number of interior vertices of 
$\Delta$. The two-dimensional case follows from 
Proposition~\ref{prop:theta-symmetry} and the 
fact (explained in Section~\ref{subsec:complexes}) 
that $h_2(\Delta) = r$.

(c) Since the coning operation leaves the 
$h$-polynomial invariant, we have $\theta(\Delta,x) 
= 0$ for every homology ball $\Delta$ which is the 
cone over a homology sphere.

(d) The coefficient of $x^2$ in $\theta(\Delta,x)$
is equal to $f^\circ_1(\Delta) - f_0(\Delta) - 
(n-2) f^\circ_0(\Delta) + n-1$ for every 
$(n-1)$-dimensional homology ball $\Delta$, where 
$f^\circ_0(\Delta)$ and $f^\circ_1(\Delta)$ are the
numbers of interior vertices and interior edges of 
$\Delta$, respectively.
\qed
\end{example}

The significance of theta polynomials stems from 
the following two theorems, the second of which 
was discovered and proven in~\cite{KMS19} 
in order to show that the local $h$-polynomial 
$\ell_V(\sd(\cC),x)$ is $\gamma$-positive for 
every regular cell decomposition $\cC$ of the 
simplex $2^V$. We recall from 
Section~\ref{subsec:bary} that $d_n(x)$, the 
$n$th derangement polynomial, is equal to the 
local $h$-polynomial of the barycentric 
subdivision of the $(n-1)$-dimensional simplex.
\begin{theorem} \label{thm:theta-formula}
For every pure simplicial complex $\Delta$ and 
every triangulation $\Delta'$ of $\Delta$,
\begin{equation} \label{eq:theta-formula}
h(\Delta', x) = \sum_{F \in \Delta} 
\theta (\Delta'_F, x) \, h(\sd(\link_\Delta(F)), 
x).
\end{equation}
\end{theorem}

\begin{theorem} \label{thm:KMS}
{\rm (\cite[Theorem~4.4]{KMS19})} 
For every triangulation $\Gamma$ of the simplex 
$2^V$,
\begin{equation} \label{eq:KMS}
\ell_V (\Gamma, x) = \sum_{F \subseteq V}
\theta (\Gamma_F, x) \, d_{|V \sm F|} (x).
\end{equation} 
\end{theorem}

\begin{proof}
[Proof of Theorem~\ref{thm:theta-formula}]
Applying Theorems~\ref{thm:sta-locality} 
and~\ref{thm:KMS} successively, changing the order 
of summation and applying 
Theorem~\ref{thm:sta-locality} once more, with 
$\Delta$ and $\Delta'$ replaced by $\link_\Delta 
(F)$ and its barycentric subdivision, 
respectively, we get

\begin{subequations}
\begin{align*}
h(\Delta', x) &= \sum_{G \in \Delta} 
\ell_G (\Delta'_G, x) \, h (\link_\Delta (G), x) \\ 
&= \sum_{G \in \Delta} \left( \, \sum_{F \subseteq G} 
\theta (\Delta'_F, x) \, d_{|G \sm F|} (x) \right) 
h (\link_\Delta (G), x) \\
&= \sum_{F \in \Delta} \theta (\Delta'_F, x) 
\left( \, \sum_{F \subseteq G \in \Delta} 
h (\link_\Delta (G), x) \, d_{|G \sm F|} (x) \right)
\\ &=  \sum_{F \in \Delta} \theta 
(\Delta'_F, x) \, h(\sd(\link_\Delta (F)), x)
\end{align*}
\end{subequations}
and the proof follows.
\end{proof}

We will say that a homology ball $\Delta$ has  
the \emph{interior vertex property} if no facet 
of $\Delta$ has all its vertices on $\partial 
\Delta$ (equivalently, if every facet of 
$\Delta$ has an interior vertex). Zero-dimensional 
balls have the interior vertex property, since 
their boundary complex equals $\{ \varnothing \}$. 
We will also say that $\Delta$ is 
\begin{itemize}
\item[$\bullet$] 
\emph{theta positive}, if $\theta(\Delta,x)$ has 
nonnegative coefficients;

\item[$\bullet$] 
\emph{theta unimodal}, if $\theta(\Delta,x)$ has 
(nonnegative and) unimodal coefficients;

\item[$\bullet$]
\emph{theta $\gamma$-positive}, if $\theta(\Delta,x)$
is $\gamma$-positive. 

\end{itemize}
A triangulation $\Delta'$ of any simplicial complex
$\Delta$ will be called \emph{theta positive}, 
\emph{theta unimodal} or \emph{theta 
$\gamma$-positive} if the restriction $\Delta'_F$ 
has the corresponding property for every $F \in 
\Delta$. 

The following statement, which is a special 
case of Stanley's monotonicity theorem 
\cite[Theorem~2.1]{Sta93}, provides a natural
sufficient condition for theta positivity.
\begin{proposition} 
\label{prop:theta-positivity} 
Let $\Delta$ be a homology ball having the 
interior vertex property.
\begin{itemize}
\item[(a)] 
{\rm (\cite{Sta93})}
$\Delta$ is theta positive.

\item[(b)] 
More generally, $\theta(\link_\Delta(F),x) \ge 0$ 
for every $F \in \partial \Delta$.
\end{itemize}
\end{proposition}

\begin{proof}
Part (a) is a direct consequence of 
\cite[Theorem~2.1]{Sta93}, applied to $\Delta$
and its subcomplex $\partial \Delta$. Part (b) 
follows from part (a) and the fact that the 
interior vertex property is inherited by the 
links in $\Delta$ of faces $F \in \partial 
\Delta$. Indeed, suppose that $E$ is a 
facet of $\link_\Delta(F)$ for some $F \in 
\partial \Delta$. Then, $E \cup F$ is a facet of 
$\Delta$ and hence it has an interior vertex, 
say $v$. Clearly, $v \in E$. Since $\partial 
\, \link_\Delta(F) = \link_{\partial \Delta} (F)$
and $v \notin \partial \Delta$, $v$ must be an 
interior vertex of $\link_\Delta(F)$. 
\end{proof}

\begin{corollary} \label{cor:min-h}
We have:
\begin{itemize}
\item[(a)] 
$h(\Delta', x) \ge h(\sd(\Delta), x)$ for every 
Cohen--Macaulay simplicial complex $\Delta$ and 
every theta positive triangulation $\Delta'$ of 
$\Delta$,

\item[(b)] 
$\ell_V(\Gamma, x) \ge \ell_V(\sd(2^V), x)$ for 
every theta positive triangulation $\Gamma$ of 
the simplex $2^V$.
\end{itemize}

\smallskip
In particular, the barycentric subdivisions 
$\sd(\Delta)$ and $\sd(2^V)$ minimize coefficientwise 
$h(\Delta', x)$ and $\ell_V(\Gamma,x)$ among all 
triangulations $\Delta'$ of $\Delta$ and $\Gamma$ of 
$2^V$, respectively, whose restrictions to all 
nonempty faces of $\Delta$ and $2^V$ have the 
interior vertex property. 
\end{corollary}

\begin{proof} 
Since $\theta (\Delta'_\varnothing, x) = 1$, one may 
rewrite Equation~(\ref{eq:theta-formula}) as 
\[ h(\Delta', x) = h(\sd(\Delta), x) + 
\sum_{F \in \Delta \sm \{ \varnothing \}} 
\theta (\Delta'_F, x) \, h(\sd(\link_\Delta (F)), x) 
   \]
and the proof of (a) follows by the nonnegativity of 
$\theta (\Delta'_F, x)$ and $h(\sd(\link_\Delta (F)), 
x)$. The proof of (b) is similar.
\end{proof}

\begin{corollary} \label{cor:unimodal-h}
\begin{itemize}
\item[(a)] 
The polynomial $h(\Delta',x)$ is unimodal, with a
peak at position $n/2$, if $n$ is even, and at 
$(n-1)/2$ or $(n+1)/2$, if $n$ is odd, for every 
$(n-1)$-dimensional Cohen--Macaulay simplicial 
complex $\Delta$ and every theta unimodal 
triangulation $\Delta'$ of $\Delta$.

\item[(b)] 
The polynomial $\ell_V (\Gamma, x)$ is unimodal
(respectively, $\gamma$-positive) for every 
theta unimodal (respectively, theta 
$\gamma$-positive) triangulation $\Gamma$ of the 
simplex $2^V$.
\end{itemize}
\end{corollary}

\begin{proof} 
Given that products of polynomials with nonnegative,
symmetric and unimodal coefficients have the same 
property \cite[Proposition~1]{Sta89}, 
Theorem~\ref{thm:theta-formula} and 
Proposition~\ref{prop:h-sd} imply that $h(\Delta',x)$ 
can be written as a sum of three polynomials with 
nonnegative, symmetric and unimodal coefficients 
and centers of symmetry $(n-1)/2$, $n/2$ and 
$(n+1)/2$, respectively. This implies part (a). 
For part (b) one uses Theorem~\ref{thm:KMS}, the 
unimodality and $\gamma$-positivity of derangement 
polynomials and the fact that products of 
$\gamma$-positive polynomials are $\gamma$-positive.
\end{proof}

\begin{corollary} \label{cor:sym-decomposition-h}
Let $\Delta$ be an $(n-1)$-dimensional simplicial
complex.
\begin{itemize}
\item[(a)] 
If $\Delta$ is a homology sphere, then 
$h(\Delta',x)$ is unimodal (respectively, 
$\gamma$-positive) for every theta unimodal 
(respectively, theta $\gamma$-positive) 
triangulation $\Delta'$ of $\Delta$.

\item[(b)] 
If $\Delta$ is Cohen--Macaulay*, then 
$h(\Delta',x)$ has a unimodal (respectively, 
$\gamma$-positive) symmetric decomposition with 
respect to $n$ for every theta unimodal 
(respectively, theta $\gamma$-positive) 
triangulation $\Delta'$ of $\Delta$.

\item[(c)] 
If $\Delta$ is a homology ball, then 
$h(\Delta',x)$ has a unimodal (respectively, 
$\gamma$-positive) symmetric decomposition 
with respect to $n-1$ for every theta unimodal 
(respectively, theta $\gamma$-positive) 
triangulation $\Delta'$ of $\Delta$.
\end{itemize}
\end{corollary}

\begin{proof} 
This follows from 
Theorem~\ref{thm:theta-formula} and known 
results about barycentric subdivisions, as was 
the case with Corollary~\ref{cor:unimodal-h}. 
For part (a), we note that $\link_\Delta (F)$
is a homology sphere of dimension $n - |F|$
for every $F \in \Delta$. As a result, 
$h(\sd(\link_\Delta(F), x)$ is symmetric, with 
nonnegative coefficients and center of symmetry 
$(n-|F|)/2$, and real-rooted by the main result 
of \cite{BW08}, thus also unimodal and 
$\gamma$-positive. Under our 
assumptions on $\Delta'$, it follows that 
all terms of the sum on the right-hand side 
of~(\ref{eq:theta-formula}) are symmetric and
unimodal (respectively, $\gamma$-positive) 
with center of symmetry $n/2$, and hence so is 
the left-hand side $h(\Delta',x)$. 

For part (b) we use instead the fact that 
$h(\sd(\link_\Delta(F), x)$ has a nonnegative
and real-rooted symmetric decomposition with 
respect to $n-|F|$ for every $F \in \Delta$. This 
is essentially due to \cite[Theorem~5.1]{BS20}; 
see \cite[Corollary~4.3~(a)]{AT21}. For part 
(c) we use the fact that either 
$h(\sd(\link_\Delta(F), x)$ is symmetric, with 
nonnegative coefficients and center of symmetry 
$(n-|F|)/2$, and real-rooted, and $\theta
(\Delta'_F,x)$ has zero constant term (if $F \in 
\inte(\Delta)$), or $h(\sd(\link_\Delta(F), x)$ 
has a nonnegative and real-rooted symmetric 
decomposition with respect to $n-|F|-1$ (if $F 
\in \partial \Delta$). The latter was shown in 
\cite[Proposition~8.1]{Ath20+}. 
\end{proof}

\begin{remark} \label{rem:g-theorem} \rm
As a consequence of the $g$-theorem, $h(\Delta,x)$
is known to be unimodal for the boundary complex 
$\Delta$ of any simplicial polytope 
\cite[Section~III.1]{StaCCA} and, more 
generally, for every triangulation $\Delta$ of a
sphere \cite{Ad18, APP21, KX22, PP20}. 
\end{remark}

\begin{remark} \label{rem:theta-corollaries} \rm
The proofs of Corollaries~\ref{cor:unimodal-h} 
and~\ref{cor:sym-decomposition-h} show that in 
their statements, $h(\Delta',x)$ and $\ell_V
(\Gamma,x)$ can be replaced by $h(\Delta',x) - 
h(\sd(\Delta),x)$ and $\ell_V(\Gamma,x) - 
\ell_V(\sd(2^V),x)$, respectively.
\end{remark}



\section{Monotonicity of theta polynomials}
\label{sec:monotone}

This section proves the monotonicity of theta 
polynomials, under mild assumptions, and explores 
its consequences on their unimodality and 
$\gamma$-positivity. We state two main theorems 
with similar proofs, based on 
Theorems~\ref{thm:sta-locality} 
and~\ref{thm:theta-formula}, respectively.

\begin{theorem} \label{thm:theta-monotone-a}
We have $\theta(\Delta', x) \ge \theta(\Delta, x)$ 
for every homology ball $\Delta$ having the 
interior vertex property and every triangulation 
$\Delta'$ of $\Delta$.
\end{theorem}

\begin{proof}
Applying Equation~(\ref{eq:sta-locality}) to the 
defining equation~(\ref{eq:theta-def}) of 
$\theta(\Delta', x)$, we get
\begin{subequations}
\begin{align*}
\theta(\Delta', x) &= 
h(\Delta', x) - h(\partial \Delta', x) \\ &= 
\sum_{F \in \Delta} \ell_F (\Delta'_F, x) \, 
h (\link_\Delta (F), x) - 
\sum_{F \in \partial \Delta} \ell_F (\Delta'_F, x) \, 
h (\link_{\partial \Delta} (F), x)
\\ &= h(\Delta, x) - h(\partial \Delta, x) 
+ \sum_{F \in \inte(\Delta)} \ell_F (\Delta'_F, x) \, 
h(\link_\Delta (F), x) + \\ & 
\sum_{F \in \partial \Delta \sm \{ \varnothing\} } 
\ell_F (\Delta'_F, x) (h(\link_\Delta (F), x) - 
h(\partial \, \link_\Delta (F), x)) \\ &= 
\theta(\Delta, x) + 
\sum_{F \in \inte(\Delta)} \ell_F (\Delta'_F, x) \, 
h(\link_\Delta (F), x) + \\ & 
\sum_{F \in \partial \Delta \sm \{ \varnothing\} } 
\ell_F (\Delta'_F, x) \, 
\theta(\link_\Delta (F), x). 
\end{align*}
\end{subequations}

\noindent    
By Proposition~\ref{prop:theta-positivity} we have 
$\theta(\link_\Delta (F),x) \ge 0$ for every 
$F \in \partial \Delta$ and the proof follows.
\end{proof}

\begin{theorem} \label{thm:theta-monotone-b}
Let $\Delta$ be a homology ball and $\Delta'$ 
be a theta positive triangulation of $\Delta$. Then, 
$\theta(\Delta', x) \ge \theta(\sd(\Delta), x)$. 
Moreover:
\begin{itemize}
\item[(a)] 
if $\Delta'$ is a theta unimodal triangulation of 
$\Delta$, then $\theta(\Delta', x)$ and 
$\theta(\Delta', x) - \theta(\sd(\Delta), x)$ have 
nonnegative and unimodal coefficients;

\item[(b)] 
if $\Delta'$ is a theta $\gamma$-positive triangulation 
of $\Delta$, then $\theta(\Delta', x)$ and 
$\theta(\Delta', x) - \theta(\sd(\Delta), x)$ are 
$\gamma$-positive.
\end{itemize}
\end{theorem}

\begin{proof}
Applying Equation~(\ref{eq:theta-formula}) to the 
defining equation~(\ref{eq:theta-def}) of 
$\theta(\Delta', x)$, we get
\begin{subequations}
\begin{align*}
\theta(\Delta', x) &= 
h(\Delta', x) - h(\partial \Delta', x) \\ &= 
\sum_{F \in \Delta} \theta_F (\Delta'_F, x) \, 
h (\sd(\link_\Delta (F)), x) - 
\sum_{F \in \partial \Delta} \theta_F (\Delta'_F, x) 
\, h (\sd(\link_{\partial \Delta} (F)), x)
\\ &= 
h(\sd(\Delta), x) - h(\sd(\partial \Delta), x) 
+ \sum_{F \in \inte(\Delta)} \theta_F (\Delta'_F, x) 
\, h(\sd(\link_\Delta (F)), x) + \\ & 
\sum_{F \in \partial \Delta \sm \{ \varnothing\} } 
\theta_F (\Delta'_F, x) (h(\sd(\link_\Delta (F)), x) 
- h(\sd(\partial \, \link_\Delta(F)), x)) \\ 
&= h(\sd(\Delta), x) - h(\partial (\sd(\Delta)), x) 
+ \sum_{F \in \inte(\Delta)} \theta_F (\Delta'_F, x) 
\, h(\sd(\link_\Delta (F)), x) + \\ & 
\sum_{F \in \partial \Delta \sm \{ \varnothing\} } 
\theta_F (\Delta'_F, x) (h(\sd(\link_\Delta (F)), x) 
- h(\partial (\sd(\link_\Delta (F)), x))) \\ 
&= \theta(\sd(\Delta), x) + 
\sum_{F \in \inte(\Delta)} \theta_F (\Delta'_F, x) 
\, h(\sd(\link_\Delta (F)), x) + \\ & 
\sum_{F \in \partial \Delta \sm \{ \varnothing\} } 
\theta_F (\Delta'_F, x) \, 
\theta(\sd(\link_\Delta (F)), x).
\end{align*}
\end{subequations}

\noindent
All claims in the statement of the proposition 
follow, since $h$-polynomials and 
$\theta$-polynomials of barycentric subdivisions of 
homology spheres and balls, respectively, are known
to be $\gamma$-positive (and, in particular, 
unimodal); see~\cite[Section~3.1.1]{Ath18} and 
Section~\ref{sec:flag}.
\end{proof}

\begin{remark} \label{rem:theta-monotone} \rm
As expected, we also have $\theta(\sd(\Delta), x) 
\ge \theta(\Delta,x)$ for every homology ball 
$\Delta$ as a consequence of the formula
\[ \theta(\sd(\Delta),x) = \sum_{i=0}^{n-1} 
\left( h_n(\Delta) + \cdots + h_{n-i}(\Delta) + 
x \, (h_n(\Delta) + \cdots + h_{i+1}(\Delta)) 
   \right) p_{n-1,i} (x), \]
where $n-1 = \dim(\Delta)$, and 
Equation~(\ref{eq:theta-Sta93}). This formula 
follows from Lemmas~3.5 and~4.2 in \cite{AT21} 
(in the special case of barycentric subdivisions).
\end{remark}

\begin{corollary} \label{cor:unimodal-local-h}
Given any triangulation $\Gamma$ of the simplex 
$2^V$, the polynomials $\ell_V (\Gamma', x)$ and 
$\ell_V (\Gamma', x) - \ell_V (\sd(\Gamma), x)$ 
are unimodal (respectively, $\gamma$-positive) 
for every theta unimodal (respectively, theta 
$\gamma$-positive) triangulation $\Gamma'$ of
$\Gamma$.  
\end{corollary}

\begin{proof}
Applying Theorem~\ref{thm:KMS} to $\Gamma'$ 
and $\sd(\Gamma)$ gives 
\begin{subequations}
\begin{align*}
\ell_V (\Gamma', x) & = \sum_{F \subseteq V}
\theta (\Gamma'_F, x) \, d_{|V \sm F|} (x), \\ 
\ell_V (\Gamma', x) - \ell_V (\sd(\Gamma), x) & =  
\sum_{F \subseteq V} \left( \theta(\Gamma'_F, x) 
- \theta(\sd(\Gamma_F),x) \right) d_{|V \sm F|} 
                   (x). 
\end{align*}
\end{subequations}

\noindent
The result follows from these formulas and 
Theorem~\ref{thm:theta-monotone-b}, since the 
latter implies the unimodality (respectively,
$\gamma$-positivity) of $\theta(\Gamma'_F,x)$ 
and $\theta(\Gamma'_F, x) - \theta(\sd(\Gamma_F),
x)$ for each $F \subseteq V$.
\end{proof}

We end this section with a discussion of the 
following general problem.

\begin{question} \label{que:monotonicity}
Let $\Delta$ and $\Delta'$ be homology balls 
such that $\Delta$ is a subcomplex of $\Delta'$.
Under what conditions does the inequality 
$\theta(\Delta', x) \ge \theta(\Delta,x)$ 
hold? Does it suffice to assume, for instance, 
that $\Delta$ and $\Delta'$ have the interior 
vertex property and the same dimension?  
\end{question}

The following theorem provides a partial 
answer to this question. The proof follows that 
of \cite[Theorem~2.1]{Sta93} and assumes 
familiarity with the basics of Stanley--Reisner 
theory~\cite[Chapter~II]{StaCCA}.

\begin{theorem} \label{thm:theta-monotone-c}
Let $\Delta$ and $\Delta'$ be homology balls 
of the same dimension, such that $\Delta$ is 
a subcomplex of $\Delta'$. If no facet of 
$\Delta'$ has all its vertices in $\partial
\Delta \cup \partial \Delta'$, then 
$\theta(\Delta', x) \ge \theta(\Delta, x)$.
\end{theorem}

\begin{proof}
By extending the field $\kk$, if necessary, we
may assume it is infinite. We consider the 
polynomial ring over $\kk$ in variables $x_1, 
x_2,\dots,x_m$ corresponding to the vertices 
of $\Delta'$, endowed with the standard grading, 
and the Stanley--Reisner rings 
$\kk[\Delta']$ and $\kk[\Delta]$. Because of 
our assumption on the facets of $\Delta'$,
just as in the proof of \cite[Theorem~2.1]{Sta93}
we may choose a linear system of parameters 
$(\eta') = (\eta'_1, \eta'_2,\dots,\eta'_n)$ 
for $\kk[\Delta']$ such that $\eta'_n$ is a 
linear combination of vertices not in $\partial
\Delta \cup \partial \Delta'$, where $n-1$ is 
the common dimension of $\Delta'$ and $\Delta$.
Then, $(\eta'_0) = (\eta'_1, 
\eta'_2,\dots,\eta'_{n-1})$ is a linear system 
of parameters for $\kk[\partial \Delta']$ and the 
images $(\eta)$ and $(\eta_0)$ of $(\eta')$ and 
$(\eta'_0)$ in $\kk[\Delta]$ and 
$\kk[\partial \Delta]$ are linear systems of 
parameters for these rings, respectively. 
Moreover, we have a diagram of (standard) graded, 
surjective homomorphisms
\[ \begin{array}{ccccccc}
    \kk[\Delta']/(\eta') & 
		\overset{\varphi'}{\rightarrow} & 
		\kk[\partial \Delta']/(\eta'_0) \\
    & & & & & & \\
    \downarrow \psi & & & & \\
    & & & & & &\\
    \kk[\Delta]/(\eta) & 
		\overset{\varphi}{\rightarrow} & 
		\kk[\partial \Delta]/(\eta_0)
   \end{array} \]
where $\varphi', \varphi$ and $\psi$ are the 
obvious maps. The kernels, say $\iI'$ and $\iI$,
of $\varphi'$ and $\varphi$ are graded ideals 
of $\kk[\Delta']/(\eta')$ and $\kk[\Delta]/
(\eta)$, respectively. Since $\Delta',
\Delta$ and their boundaries are Cohen--Macaulay,
$\iI'$ and $\iI$ have Hilbert series equal to 
$\theta(\Delta',x)$ and $\theta(\Delta,x)$,
respectively. Thus, it suffices to verify that 
$\iI \subseteq \psi(\iI')$. This is indeed the 
case, since $\iI$ is generated by the classes in
$\kk[\Delta]/(\eta)$ of squarefree monomials 
which correspond to interior faces of $\Delta$ 
and every such face is also an interior face of 
$\Delta'$.
\end{proof}

\begin{remark} 
\label{rem:monotonicity-que} \rm
The assumptions that $\Delta$ and $\Delta'$ 
have the same dimension and nonnegative theta
polynomials do not suffice to guarantee that 
$\theta(\Delta', x) \ge \theta(\Delta,x)$. 
For example, let $\Delta'$ be any 
$(n-1)$-dimensional homology ball which has 
at least two facets and a vertex $v$ which 
belongs to a unique facet of $\Delta'$, where 
$n \ge 4$, and let $\Delta$ be obtained from 
$\Delta'$ by removing all faces containing $v$. 
Then, $\Delta$ is also an $(n-1)$-dimensional 
homology ball and $\theta(\Delta,x) = \theta
(\Delta',x) + (x^2 + x^3 + \cdots + x^{n-2})$. 
Note that, in this situation, $\Delta'$ does 
not have the interior vertex property.

One may pick $\Delta'$ so that $\theta(\Delta',
x) \ge 0$, for instance, by letting $\Delta'$ 
be the 4-fold edgewise subdivision of the 
three-dimensional simplex, in which case 
$\theta(\Delta',x) = 0$ (see Section~3.2 and 
Example~7.2 of \cite{Ath20+}). 
\end{remark}

\section{Unimodality and gamma-positivity}
\label{sec:flag}

This section investigates the unimodality and 
$\gamma$-positivity of theta polynomials. These
questions are naturally raised by the results
of the previous sections and are closely related 
to the unimodality and $\gamma$-positivity of 
the symmetric decompositions of $h$-polynomials
of homology balls. Indeed, given an 
$(n-1)$-dimensional homology ball $\Delta$, since 
$h(\partial \Delta, x)$ and $\theta(\Delta,x)$
are symmetric polynomials with centers of symmetry 
$(n-1)/2$ and $n/2$, respectively, and the latter 
has zero constant term, the expression
\begin{equation} \label{eq:sym-decomposition-h}
h(\Delta,x) = h(\partial \Delta,x) + 
   x \, \theta(\Delta,x) / x 
\end{equation}
is the symmetric decomposition of $h(\Delta,x)$
with respect to $n-1$. Thus, the unimodality 
(respectively, $\gamma$-positivity) of this 
symmetric decomposition of $h(\Delta,x)$ is 
equivalent to the unimodality (respectively, 
$\gamma$-positivity) of $h(\partial \Delta,x)$ 
and $\theta(\Delta,x)$.

\subsection{Unimodality} Because of 
Equation~(\ref{eq:theta-Sta93}), the unimodality 
of $\theta(\Delta,x)$ is equivalent to the 
inequalities 
\begin{equation} \label{eq:top-heavy}
h_i(\Delta) \le h_{n-1-i}(\Delta), \ \ \ \ \ \
0 \le i \le (n-1)/2, 
\end{equation}
where $n-1 = \dim(\Delta)$. Moreover, since 
$h(\partial \Delta,x) = h(\Delta,x) - \theta
(\Delta,x)$, Equation~(\ref{eq:theta-Sta93}) is 
equivalent to the formulas 
\[ h_i(\partial \Delta) = h_0(\Delta) + h_1
   (\Delta) + \cdots + h_i(\Delta) - h_n(\Delta) - 
	 h_{n-1}(\Delta) - \cdots - h_{n-i}(\Delta) \]
for $0 \le i \le n-1$, where $h_n(\Delta) = 0$.
Hence, the unimodality of $h(\partial \Delta, x)$ 
is equivalent to the inequalities $h_i(\Delta) 
\ge h_{n-i}(\Delta)$ for $1 \le i \le \lfloor 
n/2 \rfloor$ and the unimodality of the symmetric 
decomposition (\ref{eq:sym-decomposition-h}) 
(meaning, the unimodality of both $h(\partial 
\Delta, x)$ and $\theta(\Delta,x)$) is equivalent 
to the inequalities 
\begin{equation} \label{eq:def-alt-increasing}
  h_0(\Delta) \le h_{n-1}(\Delta) \le h_1(\Delta)  
  \le h_{n-2}(\Delta) \le \cdots \le 
	h_{\lfloor n/2 \rfloor}(\Delta), 
\end{equation}
sometimes referred to as the \emph{alternatingly 
increasing property} for $h(\Delta)$.

Given that the nonnegativity of $\theta(\Delta,x)$
is guaranteed by the interior vertex property (see
Proposition~\ref{prop:theta-positivity}), it seems 
natural to expect that its unimodality will be valid
under the stronger assumption that $\partial \Delta$ 
is an induced subcomplex of $\Delta$. Indeed, the 
almost strong Lefschetz property \cite{KN09} for 
$\Delta$ over a field $\kk$ implies the existence 
of an injective map from a $\kk$-vector space of 
dimension $h_i(\Delta)$ to a $\kk$-vector space of 
dimension $h_j(\Delta)$ for all $0 \le i \le j \le 
n-1-i$. In particular, it implies 
the inequalities~(\ref{eq:top-heavy}) and that
\begin{equation} \label{eq:half-increasing}
h_0(\Delta) \le h_1(\Delta) \le \cdots \le 
h_{\lfloor (n-1)/2 \rfloor}(\Delta). 
\end{equation}
The almost strong Lefschetz property over an 
infinite field $\kk$ for a triangulation 
$\Delta$ of a ball follows from a result 
of Adiprasito and Yashfe~\cite[Theorem~50]{AY20} 
and the Lefschetz property for homology spheres 
over $\kk$, under the assumption that $\partial 
\Delta$ is an induced subcomplex of $\Delta$. 
Thus, in view of the results of \cite{PP20} in
characteristic two, the following theorem holds.
\begin{theorem} [\cite{AY20}] \label{thm:AY20}
The polynomial $\theta(\Delta,x)$ is unimodal 
for every triangulation $\Delta$ of a ball, 
such that $\partial \Delta$ is an induced 
subcomplex of $\Delta$.
\end{theorem}

The following example shows that the condition 
that $\partial \Delta$ is an induced subcomplex of 
$\Delta$ cannot be replaced by the interior 
vertex property. 
\begin{example} \label{ex:conj-unimodality} \rm
Let $\Gamma$ be the triangulation of the 
three-dimensional ball having two facets $F = 
\{a,b,c,d\}$ and $G = \{b,c,d,e\}$. Let $\Delta$ 
be obtained from $\Gamma$ by two stellar 
subdivisions on these facets, meaning that one 
adds two new vertices $u, v$ and the unions of 
the proper subsets of $F$ and $G$ with $\{u\}$ 
and $\{v\}$, respectively, to obtain $\Delta$ 
from $\Gamma$. Then, 
\begin{subequations}
\begin{align*}
h(\Delta,x)  & = 1 + 3x + 2x^2 + 2x^3 \\ 
h(\partial \Delta) = h(\partial \Gamma) & = 
1 + 2x + 2x^2 + x^3, 
\end{align*}
\end{subequations}
so that $\theta(\Delta,x) = x + x^3$ is not 
unimodal. Note that $\partial \Delta$ is not an 
induced subcomplex of 
$\Delta$, since $\{b, c, d\}$ is an interior 
face of $\Delta$ having all its vertices on 
$\partial \Delta$. On the other hand, all facets 
of $\Delta$ have an interior vertex, namely $u$ 
or $v$. 
\qed
\end{example}

\begin{remark} \rm
Stanley proved~\cite[Theorem~5.2]{Sta92} that the 
local $h$-polynomial $\ell_V(\Gamma,x)$ is unimodal
for every regular triangulation $\Gamma$ of the 
simplex $2^V$. Moreover, he conjectured
\cite[Conjecture~5.4]{Sta92} that the same holds 
for every quasigeometric simplicial subdivision of 
$2^V$. Although the conjecture fails at this level 
of generality \cite[Example~3.4]{Ath12}, it is open 
for the class of geometric triangulations studied 
in this paper (see \cite[Question~3.5]{Ath12} for a 
class of topological simplicial subdivisions of the 
simplex for which it could be true). 
Theorem~\ref{thm:AY20}, combined with 
Theorem~\ref{thm:KMS}, shows that $\ell_V(\Gamma,x)$ 
is unimodal for every triangulation $\Gamma$ of the 
simplex $2^V$ such that for every $F \subseteq V$, 
the restriction of $\Gamma$ to $\partial (2^F)$ is 
an induced subcomplex of $\Gamma$.
\qed
\end{remark}

\subsection{Gamma-positivity} A simplicial complex 
$\Delta$ is said to be \emph{flag} if every set $F$
of vertices of $\Delta$ such that $\{u,v\} \in \Delta$ 
for all $u, v \in F$ is a face of $\Delta$. The 
$h$-polynomials of simplicial complexes seem to 
behave well with respect to $\gamma$-positivity under
the flagness condition. For example, $h$-polynomials
of flag homology spheres \cite[Conjecture~2.1.7]{Ga05} 
and local $h$-polynomials of flag triangulations of 
simplices \cite[Conjecture~5.4]{Ath12} are conjectured 
to be $\gamma$-positive. In the spirit of these 
conjectures, it seems natural to formulate the 
following flag conjectural analogue of 
Theorem~\ref{thm:AY20}. 
\begin{conjecture} 
\label{conj:theta-gamma-positivity}
The polynomial $\theta(\Delta,x)$ is 
$\gamma$-positive for every flag homology ball
$\Delta$, such that $\partial \Delta$ is an 
induced subcomplex of $\Delta$.
\end{conjecture}

The following example shows that, once again, the 
condition that $\partial \Delta$ is an induced 
subcomplex of $\Delta$ cannot be replaced by the
interior vertex property.
\begin{example} \label{ex:conj-gamma-positivity} 
\rm Consider the boundary complexes $\Delta_1$ 
and $\Delta_2$ of two octahedra, which are glued 
along a common facet $F$, and let $\Delta$ be the 
union of the cones of $\Delta_1$ and $\Delta_2$ on 
two vertices $u_1$ and $u_2$. Thus, $\Delta$ is
a three-dimensional flag simplicial ball with eleven
vertices and sixteen facets which has the interior 
vertex property (every facet contains either $u_1$ 
or $u_2$). The boundary $\partial \Delta$ has nine 
vertices and fourteen facets, and
\begin{subequations}
\begin{align*}
h(\Delta,x)  & = 1 + 7x + 6x^2 + 2x^3 \\ 
h(\partial \Delta) & = 1 + 6x + 6x^2 + x^3, 
\end{align*}
\end{subequations}
so that $\theta(\Delta',x) = x + x^3$ is not 
$\gamma$-positive (not even unimodal). Since $F$ is 
an interior face of $\Delta$ having all its vertices 
on $\partial \Delta$, the latter is not an induced 
subcomplex of $\Delta$. \qed
\end{example}

Conjecture~\ref{conj:theta-gamma-positivity} turns
out to be equivalent to a strengthening of Gal's 
conjecture~\cite[Conjecture~2.1.7]{Ga05}, 
proposed recently by Chudnovsky and 
Nevo~\cite{CN20} in two equivalent forms (named
the Link Conjecture and the Equator Conjecture), 
as we now show. We recall that the 
\emph{$\gamma$-polynomial} of an $(n-1)$-dimensional 
homology sphere $\Delta$ is defined as $\gamma
(\Delta,x) = \sum_{i=0}^{\lfloor n/2 \rfloor}
\gamma_i(\Delta) x^i$, where $h(\Delta,x) = 
\sum_{i=0}^{\lfloor n/2 \rfloor} \gamma_i(\Delta) 
x^i (1+x)^{n-2i}$.
\begin{proposition} 
\label{prop:theta-gamma-positivity-a}
For every positive integer $n$, the following 
statements are equivalent:
\begin{itemize}
\item[(i)] 
Conjecture~\ref{conj:theta-gamma-positivity} 
holds for all homology balls of dimension $n-1$.

\item[(ii)] 
{\rm (\cite[Conjecture~1.2]{CN20})}
We have $\gamma(\Delta,x) \ge \gamma
(\link_\Delta(v),x)$ for every $(n-1)$-dimensional 
flag homology sphere $\Delta$ and every vertex $v$ 
of $\Delta$.

\item[(iii)]
{\rm (\cite[Conjecture~1.3]{CN20})}
We have $\gamma(\Delta,x) \ge \gamma(\Gamma,x)$ 
for every $(n-1)$-dimensional flag homology 
sphere $\Delta$ and every $(n-2)$-dimensional 
flag homology sphere $\Gamma$ which is an 
induced subcomplex of $\Delta$.
\end{itemize}
\end{proposition}

\begin{proof} 
The equivalence ${\rm (ii)} \Leftrightarrow 
{\rm (iii)}$ has already been proven in 
\cite[Proposition~3.1]{CN20}. To prove that 
${\rm (i)} \Rightarrow {\rm (ii)}$, we assume that 
(i) holds and consider an $(n-1)$-dimensional 
flag homology sphere $\Delta$, along with a 
vertex $v$. Then, $\Gamma := \Delta \sm v$ is 
an $(n-1)$-dimensional homology ball with 
boundary $\partial \Gamma = \link_\Delta(v)$.
The flagness of $\Delta$ implies that $\partial
\Gamma$ is an induced subcomplex of $\Gamma$
and that $\Gamma$ and $\partial \Gamma$ are 
both flag. Thus, $\theta(\Gamma, x)$ is 
$\gamma$-positive. Since, as an easy consequence 
of the definition of the $h$-polynomial, 
$h(\Delta,x) = h(\Gamma,x) + x \, 
h(\link_\Delta(v),x)$, we have
\begin{subequations}
\begin{align*} 
\theta(\Gamma,x) & = h(\Gamma,x) - h(\partial 
\Gamma,x) = h(\Gamma,x) - h(\link_\Delta(v),x) 
\\ & = h(\Delta,x) - (1+x) \, 
       h(\link_\Delta(v),x).
\end{align*}
\end{subequations}
Hence, the $\gamma$-positivity of $\theta(\Delta,x)$
exactly means that $\gamma(\Delta,x) \ge 
\gamma(\link_\Delta(v),x)$.

The proof of ${\rm (ii)} \Rightarrow {\rm (i)}$ 
is similar. Given an $(n-1)$-dimensional flag 
homology ball $\Gamma$ with induced boundary 
$\partial \Gamma$, one adds the cone of 
$\partial \Gamma$ over a new vertex $v$ to 
$\Gamma$ to obtain an $(n-1)$-dimensional flag 
homology sphere $\Delta$ with $\link_\Delta(v) 
= \partial \Gamma$. Because $\partial \Gamma$ is 
an induced subcomplex of $\Gamma$, the sphere 
$\Delta$ is also flag and hence $\gamma
(\Delta,x) \ge \gamma(\link_\Delta(v),x)$ by 
(ii). As already shown, the latter inequality 
is equivalent to the $\gamma$-positivity of 
$\theta(\Gamma,x)$. 
\end{proof}

One may also state 
Conjecture~\ref{conj:theta-gamma-positivity} in 
the language of symmetric decompositions and place
it within the phenomena of `nonsymmetric 
$\gamma$-positivity' discussed in 
\cite[Section~5.1]{Ath18}. 
\begin{proposition} 
\label{prop:theta-gamma-positivity-b}
For every positive integer $d$, the following 
statements are equivalent:
\begin{itemize}
\item[(i)] 
Conjecture~\ref{conj:theta-gamma-positivity}
holds in all dimensions less than $d$. 

\item[(ii)] 
For all $1 \le n \le d$, the polynomial 
$h(\Gamma,x)$ has a $\gamma$-positive symmetric 
decomposition with respect to $n-1$ for every 
$(n-1)$-dimensional flag homology ball $\Gamma$, 
such that $\partial \Gamma$ is an induced 
subcomplex of $\Gamma$.
\end{itemize}
\end{proposition}
 
\begin{proof} 
As already discussed, the $\gamma$-positivity of 
the symmetric decomposition 
of $h(\Gamma,x)$ with respect to $n-1$ is 
equivalent to the $\gamma$-positivity of 
$\theta(\Gamma,x)$ and $h(\partial \Gamma,x)$. 
Thus, we only need to verify that condition (i) 
implies the $\gamma$-positivity of the 
$h$-polynomials of flag homology spheres of 
dimension less than $d$. This follows from 
Proposition~\ref{prop:theta-gamma-positivity-a}
since, as noted in~\cite{CN20}, the validity of 
condition (ii) there for all $n \le d$ implies the
validity of Gal's conjecture in all dimensions less
than $d$ by induction on the dimension.
\end{proof}

\begin{remark} \rm
Conjecture~\ref{conj:theta-gamma-positivity} would
imply, in view of Theorem~\ref{thm:KMS}, that 
$\ell_V(\Gamma,x)$ is $\gamma$-positive for every 
flag triangulation $\Gamma$ of the 
simplex $2^V$ such that for every $F \subseteq V$, 
the restriction of $\Gamma$ to $\partial (2^F)$ is 
an induced subcomplex of $\Gamma$ (thus verifying 
\cite[Conjecture~5.4]{Ath12} in this situation). 
\qed
\end{remark}

Conjecture~\ref{conj:theta-gamma-positivity} is 
true, among other special cases, when $\Delta$ is:
\begin{itemize}
\item[$\bullet$] 
the barycentric subdivision of any regular
cell decomposition of a ball, by results of Karu 
and Ehrenborg~\cite{EK07} (see 
\cite[Section~4]{KMS19} for a detailed explanation
and \cite[Theorem~3.9]{Ath18}), 

\item[$\bullet$] 
the $r$-fold edgewise subdivision (for $r \ge n$)
and the $r$-colored barycentric subdivision of any 
$(n-1)$-dimensional homology ball; see part (b) of 
Corollaries~5.1 and 5.4 in~\cite{AT21},

\item[$\bullet$] 
the antiprism triangulation of any $(n-1)$-dimensional 
homology ball; see Section~\ref{sec:app}.
\end{itemize}

The stronger property that $\theta(\Delta,x)$ has 
only real roots is proven for edgewise subdivisions 
and $r$-colored barycentric subdivisions
in~\cite{AT21}. Further evidence in favor of 
Conjecture~\ref{conj:theta-gamma-positivity} can be 
found in~\cite{CN20}.

\section{Antiprism triangulations}
\label{sec:app}

The antiprism triangulation of a simplicial complex 
$\Delta$, denoted by $\sd_\aA(\Delta)$, was 
first considered by Izmestiev and Joswig~\cite{IJ03} 
in the context of branched coverings of manifolds. 
As a simplicial complex, it may be defined to have 
vertices the pointed faces $(F,v)$ of $\Delta$, 
meaning pairs of faces $F \in \Delta$ and vertices 
$v \in F$, and faces the sets consisting of the 
pointed faces $(F_1,v_1), (F_2,v_2),\dots,(F_m,v_m)$
of $\Delta$, such that 
\begin{itemize}
\item[$\bullet$] 
$F_1 \subseteq F_2 \subseteq \cdots \subseteq F_m$
and 

\item[$\bullet$] 
$F_i \subsetneq F_j \Rightarrow v_j \in F_j \sm 
F_i$, for $i < j$. 
\end{itemize}
This simplicial complex naturally triangulates 
$\Delta$; the carrier of the face consisting of
$(F_1,v_1), (F_2,v_2),\dots,(F_m,v_m)$, as above, 
is the face $F_m \in \Delta$. 

The combinatorial properties of antiprism 
triangulations were studied systematically 
in~\cite{ABK22} and shown to have similarities to 
those of barycentric subdivision, but often to be 
more challenging
to analyze. For example, it is an open problem 
\cite[Conjecture~1.1]{ABK22} to decide whether the
$h$-polynomial of $\sd_\aA(\Delta)$ is real-rooted
for every Cohen--Macaulay simplicial complex 
$\Delta$. The key property that 
$\theta(\sd_\aA(2^V),x)$ has 
only real roots (and hence is unimodal and 
$\gamma$-positive) for every simplex $2^V$ was 
proven in~\cite{ABK22} (see Theorem~5.2 there).
Given that, we can reprove the main result 
of~\cite{ABK22} about the unimodality of 
$h(\sd_\aA(\Delta),x)$ and answer in the 
affirmative some of the questions posed 
in~\cite[Section~8]{ABK22} (see Remarks~3 
and~4 there). Part (b) of the following 
corollary verifies Gal's conjecture in a new 
special case, namely that of antiprism 
triangulations of homology spheres.
\begin{corollary} \label{cor:h-antiprism}
The polynomial $h(\sd_\aA(\Delta), x)$:  
\begin{itemize}
\item[(a)] 
{\rm (cf. \cite[Theorem~1.3]{ABK22})}
is unimodal, with a peak at position $n/2$, 
if $n$ is even, and at $(n-1)/2$ or $(n+1)/2$, 
if $n$ is odd, for every Cohen--Macaulay 
$(n-1)$-dimensional simplicial complex 
$\Delta$;

\item[(b)] 
is $\gamma$-positive for every 
$(n-1)$-dimensional homology sphere $\Delta$;

\item[(c)]
has a $\gamma$-positive symmetric decomposition 
with respect to $n$ for every $(n-1)$-dimensional 
Cohen--Macaulay* simplicial complex $\Delta$;

\item[(d)]
has a $\gamma$-positive symmetric decomposition with 
respect to $n-1$ for every $(n-1)$-dimensional 
homology ball $\Delta$.
\end{itemize}
\end{corollary}

\begin{proof} 
This follows from Corollary~\ref{cor:unimodal-h} 
(a), Corollary~\ref{cor:sym-decomposition-h} and 
the theta unimodality and theta
$\gamma$-positivity of antiprism triangulations
of simplicial complexes, established 
in~\cite{ABK22}.
\end{proof}

\begin{corollary} \label{cor:ell-antiprism}
The polynomial $\ell_V(\sd_\aA(\Gamma), x)$ is 
$\gamma$-positive for every triangulation $\Gamma$ 
of the simplex $2^V$.  
\end{corollary}

\begin{proof} 
This follows from 
Corollary~\ref{cor:unimodal-local-h} and the 
theta $\gamma$-positivity of antiprism 
triangulations of simplicial complexes, established 
in~\cite{ABK22}.
\end{proof}

\begin{question} \label{que:real-sd-antiprism}
Is $\ell_V(\sd(\Gamma), x)$ real-rooted for every 
triangulation $\Gamma$ of the simplex $2^V$? Is 
$\ell_V(\sd_\aA(\Gamma), x)$ real-rooted for every 
triangulation $\Gamma$ of $2^V$? 
\end{question}

\section*{Acknowledgments}

The author wishes to thank Satoshi Murai and 
Isabella Novik for pointing out that 
Theorem~\ref{thm:AY20} follows from 
\cite[Theorem~50]{AY20} and the $g$-theorem and
the anonymous referee for useful comments.

\end{document}